\documentclass[12pt]{amsart}

\usepackage{amsmath,amssymb,amsthm}
\usepackage[bookmarksopen=true,
            bookmarksnumbered=true,
            colorlinks=true,
            linkcolor=blue,
            pdftitle={Graded Betti numbers of powers},
            pdfsubject={Commutative Algebra},
            pdfkeywords={Asymptotic linearity, regularity, multigraded regularity, Betti number},
            pdfpagemode=UseOutlines,
            pdfproducer={Latex with hyperref},
            letterpaper=true,
            pdfcreator={latex->dvips->ps2pdf}]{hyperref}
\usepackage{graphicx,epsfig}
\usepackage{enumerate}
\usepackage{xypic}
\usepackage{tikz}
\usepackage{multicol}
\usetikzlibrary{calc}
  \input{xy}
  \xyoption{all}
\usepackage{url}

\numberwithin{equation}{section}

\theoremstyle{plain}
\newtheorem{theorem}{Theorem}[section]

\newtheorem{lemma}[theorem]{Lemma}

\newtheorem{corollary}[theorem]{Corollary}

\newtheorem{proposition}[theorem]{Proposition}

\theoremstyle{definition}
\newtheorem{example}[theorem]{Example}
\newtheorem{remark}[theorem]{Remark}
\newtheorem{definition}[theorem]{Definition}

\def\0{{\bf 0}}

\def\G{\mathbb{G}}

\def\NN{\mathbb{N}}

\def\R{\mathcal{R}}

\def\ZZ{\mathbb{Z}}
\def\RR{\mathbb{R}}
\def\QQ{\mathbb{Q}}

\def\P{\mathcal{P}}

\def\x{{\bf x}}

\DeclareMathOperator{\HF}{HF}

\DeclareMathOperator{\supp}{Supp}

\DeclareMathOperator{\tor}{Tor}

\DeclareMathOperator{\pos}{Pos}

\DeclareMathOperator{\en}{end}

\begin{document}

\author{Amir Bagheri}
\author{Kamran Lamei}

\title{Graded Betti numbers of powers of ideals}

\thanks{The research of Amir Bagheri was in part supported by a grant from IPM (No. 93130024).}
\begin{abstract}

Using the concept of vector partition functions,
we investigate the asymptotic behavior of
graded Betti numbers of powers of homogeneous ideals in a polynomial
ring over a field.
Our main results state that if the polynomial ring is equipped with a positive $\ZZ$-grading, then the Betti
numbers of powers of ideals are encoded by finitely many polynomials.

More precisely, in the case of $\ZZ$-grading, $\ZZ^2$
can be splitted into  a finite number of regions such that each region corresponds to a polynomial that
depending to the degree $(\mu, t)$,
$\dim_k \left(\tor_i^S(I^t, k)_{\mu} \right)$
is equal to one of these polynomials  in $(\mu, t)$.
This refines, in a graded situation, the result of Kodiyalam on Betti numbers of powers of ideals in \cite{Ko1}.

Our main statement treats the case of a power products of homogeneous ideals  in a $\ZZ^d$-graded  algebra, for a positive grading, in the sense of \cite{ms}.

\end{abstract}

\maketitle


\section{Introduction}
The study of homological invariants of powers of ideals goes back, at least,
 to the work of Brodmann in the 70's and attracted a lot of attention
 in the last two decades.

One of the most important results in this area is the result on the
asymptotic linearity of
Castelnuovo-Mumford regularity obtained by
Kodiyalam \cite{Ko} and  Cutkosky, Herzog and Trung \cite{CHT}, independently.
The proof of Cutkosky, Herzog and Trung
further describes the eventual linearity in $t$ of $\en\left(\tor_i^S(I^t, k)\right) :=$ max$\{\mu | \tor_i^S(I^t, k)_{\mu} \neq 0 \} $.

 It is natural to concern the asymptotic behavior
of Betti numbers $\beta_i (I^t):=\dim_k \tor_i^S(I^t, k)$ as $t$ varies. In  \cite{NR}, Northcott and Rees already investigated the asymptotic behavior of
$\beta_1^k(I^t)$. Later, using the Hilbert-Serre theorem,
Kodiyalam \cite[Theorem 1]{Ko1} proved that for any non-negative integer $i$ and
sufficiently large $t$, the $i$-th Betti number, $\beta_i^{k}(I^t)$,
is a polynomial $Q_i$ in $t$ of degree at most the analytic spread of $I$ minus one.

Recently, refining the result of \cite{CHT} on $\en\left(\tor_i^S(I^t, k)\right)$, \cite{BCH} gives a precise picture of the set of degrees $\gamma$
such that $\tor_i^S(I^t, A)_\gamma \not= 0$ when $t$ runs over $\NN$.
In \cite{BCH}, the authors consider  a polynomial ring $S=A[x_1, \ldots, x_n]$ over a Noetherian ring $A$ graded by a finitely
generated abelian group $G$ (see \cite[Theorem 4.6]{BCH}).

When $A=k$ is a field and the ideal is generated in a single degree $d\in G$, it is  proved that for any
$\gamma$ and any $j$, the function
 $$\dim_k\tor_i^S(I^t, k)_{\gamma +td}$$
is a  polynomial in $t$ for $t\gg 0$ (see \cite[Theorem 3.3]{BCH} and \cite{Si}).

We are here interested in the behavior of $\dim_k\tor_i^S(I^t, k)_{\gamma}$
 when $I$ is an arbitrary graded ideal and
$S=k[x_1, \ldots, x_n]$ is a $\ZZ^p$-graded polynomial
ring over a field $k$, for a positive grading in the sense of \cite{ms}.

In the case of a positive $\ZZ$-grading, our result takes the following form:

\textbf{Theorem}(See Theorem \ref{main res}).
\emph{Let $S=k[x_1, \ldots, x_n]$ be a positively graded polynomial ring over
a field $k$ and let $I$ be a homogeneous ideal in $S$.
}
\emph{There exist,
$t_0,m,D\in \ZZ$, linear functions $L_i(t)=a_i t+b_i$,
for $i=0,\ldots ,m$, with $a_i$ among the degrees
of the minimal generators of $I$ and $b_i\in \ZZ$, and polynomials $Q_{i,j}\in \QQ [x,y]$ for
$i=1,\ldots ,m$ and $j\in 1,\ldots ,D$, such that, for $t\geq t_0$,
}\\
\emph{(i) $L_i(t)<L_j(t)\ \Leftrightarrow\ i<j$,\\
(ii) If $\mu <L_0(t)$ or $\mu >L_m(t)$, then $\tor_i^S(I^t, k)_{\mu}=0$.\\
(iii)  If $L_{i-1} (t)\leq \mu \leq L_{i}(t)$ and
$a_i t-\mu \equiv j\mod (D)$, then
$$\dim_k\tor_i^S(I^t, k)_{\mu}=Q_{i,j}(\mu ,t).
$$}

Our general result, Theorem \ref{tor-general case}, involves a finitely
generated graded module $M$ and a finite collection of graded ideals
$I_1,\ldots ,I_s$. The grading is a positive $\ZZ^p$-grading and a
special type of finite decomposition of $\ZZ^{p+s}$ is described in
such a way that in each region $\dim_k (\tor_i^S(MI_1^{t_1}...I_s^{t_s}, k)_{\gamma })$
is a polynomial in $(\gamma ,t_1,\ldots ,t_s)$ (i. e. is a quasi-polynomial) with respect to a lattice defined in terms of the
degrees of generators of the ideals.

The central object in this study is the Rees modification $M\R_I$. This graded
object admits a graded free resolution over a polynomial extension of $S$
from which we deduce a  $\ZZ^{p+s}$-grading  on Tor modules as in
\cite{BCH}. Investigating Hilbert series of modules for such a grading,
using vector partition functions, leads to our results.

This paper, is organized as follows. In the next Section, we provide
some definitions and terminology that we will need.
In Section 3, we will discuss Hilbert functions of non-standard
graded rings and in the last section, we prove the main theorem.

\section{Preliminaries} \label{sec.prel}

In this section, we are going to collect some necessary notations and terminology used in the paper.
For basic facts in commutative algebra, we refer the reader to \cite{Ei, bh}.

\subsection{Hilbert series}
In this part we will recall the definition and some important properties of Hilbert functions and Hilbert series.

Let $S=k[x_1, \ldots, x_n]$ be a polynomial ring over field $k$. We first make clear our definition of grading.

\begin{definition}
Let $G$ be an abelian group.  A $G$-grading of $S$ is a group homomorphism $\deg : \ZZ^n \longrightarrow G $. We set
$\deg(x^{u}) :=\deg(u)$ for a monomial $x^{u}= x_{1}^{u_1}... x_{n}^{u_n} \in S $. An element
$\sum c_u x^{u}\in S$ is homogeneous of degree $\mu\in G$ if $\deg (u)=\mu$ whenever $c_u\not= 0$.
 The set of elements of degree $\mu$ in $S$, $S_\mu$, is called the homogeneous component of $S$ in degree $\mu$.
The grading is called positive if furthermore $G$ is torsion-free and $S_0=k$.

We also recall that for any $\mu\in G$, $S(\mu)$ is again a $G$-graded ring such that $S(\mu)_n=S_{\mu+n}$. Obviously it is isomorphic to $S$ and it is obtained just by a shift on $S$.
 \end{definition}
From now on, $S$ will be a positively graded polynomial ring.
Criterions of positivity are  given in \cite[8.6]{ms}. One defines graded ideals and modules similarly to the classical $\ZZ$-graded case.
When $G=\ZZ^d$ and the grading is positive, (generalized) power series are associated to finitely generated graded modules:

\begin{remark}
By \cite[8.8]{ms},  if $S$ is positively graded by $\ZZ^d$,  then the semigoup $Q=\deg(\NN^n)$ can be embedded in $\NN^d$. Hence,  after such a change of embedding,
the above Hilbert series are  power series in the usual sense.
\end{remark}

Let $M$ be a finitely generated $\ZZ^d$-graded $S$-module. It admits a
finite minimal graded free $S$-resolution
$$\mathbb{F}_\bullet: 0\rightarrow F_u \rightarrow \ldots \rightarrow F_1\rightarrow F_0\rightarrow M\rightarrow 0.$$
Writing $$F_i=\oplus_\mu S(-\mu)^{\beta_{i,\mu}(M)},$$ the minimality
of resolution shows that  $\beta_{i,\mu}(M)=\dim_k\left(\tor_i^S(M,k)\right)_\mu$,
as the maps of $\mathbb{F}_\bullet\otimes_S k$ are zero.
We also recall that the support of a $\ZZ^d$-graded module $N$ is
$$\supp_{\ZZ^d} (N):=\{\mu\in \ZZ^d | N_\mu\neq0\}.$$
\begin{definition}
The Hilbert function of a finitely generated module $M$ over a positively graded polynomial ring is the map
$$\begin{array}{llll}
HF(M; -):&\ZZ^d&\longrightarrow& \NN\\
&\mu&\longmapsto&\dim_k(M_\mu).
\end{array}$$
Moreover, the Hilbert series of $M$ is  the power series $$H(M;{\bold t})=\sum_{{\bf \mu}\in \ZZ^d}\dim_k(M_{\mu}){\bold t}^{\mu}$$
where ${\bold t}:=(t_1, \ldots, t_d)$. Also we write ${\bold t}\gg 0$ if for all $i=1, \ldots, d$ we have $t_i\gg 0$.
\end{definition}
\begin{proposition}\label{Hilbert nonstandard}
Let $S=k[x_1, \ldots, x_n]$ be a positively graded $\ZZ^d$-graded polynomial ring over a field $k$.
Then the followings hold.
\begin{enumerate}
\item
The Hilbert series of $S$  is the development in power series of the rational function
   $$H(S(-\mu ); {\bold t})= \frac{{\bold t}^\mu}{\prod_{i=1}^n(1-{\bold t}^{\mu_i})}$$
where $\mu_i = \deg (x_i)$.

\item If $M$ is a finitely generated graded $S$-module, setting $\Sigma_M:=\cup_{\ell}\supp_{\ZZ^d}(\tor^R_\ell (M,k))$ and
$$
\kappa_M ({\bold t}):=\sum_{a\in \Sigma_M}\left( \sum_\ell (-1)^\ell \dim_k (\tor^R_\ell (M,k))_a\right) {\bold t}^a,
$$
one has $H(M;{\bold t}) = \kappa_M (t)H(S; {\bold t})$.
\end{enumerate}
\end{proposition}
\begin{proof}
(1) is a simple calculation  and (2) follows from the existence of graded minimal free resolutions.
\end{proof}

\subsection{Vector partition function}
We first recall the definition of quasi-polynomials. Let $d\geq 1$ and $\Lambda$ be a
lattice in $\ZZ^d$.
\begin{definition}\cite{Bar}
A function $f:\ZZ^d\to\QQ$ is a quasi-polynomial with respect to $\Lambda$
if there exists a list of polynomials $Q_i\in\QQ[T_1, \ldots, T_d]$ for
$i\in \ZZ^d /\Lambda$ such that $f(s)=Q_i(s)$ if $s\equiv i\mod \Lambda$.
\end{definition}
In fact, a quasi-polynomial is a polynomial that its coefficients are replaced by periodic functions.

Notice that $\ZZ^d/\Lambda$ has $\vert \det (\Lambda )\vert $ elements, and when $d=1$, then $\Lambda =q\ZZ$ for some $q>0$. In this case $f$ is
 called a quasi-polynomial of period $q$.


\begin{definition}
Let  $A=(a_{i,j})$ be a $d\times n$-matrix of rank $d$ with entries in $\NN$.
Let $a_j:=(a_{1j}, \ldots, a_{dj})$ be the $j$-th column of $A$ and the mapping $\varphi$ acts as below:
$$\begin{array}{llll}
\varphi_A:&\NN^d&\longrightarrow&\NN\\
&u&\longrightarrow&\#\left\{ \lambda\in \NN^n \vert A.\lambda =u \right\}.
\end{array}$$
Then $\varphi$ is called the \textbf{vector partition function} corresponding to the matrix $A$.
Equivalently,  $\varphi_A(u)$ is the coefficient
of ${\bold t}^u$ in the formal power series $\prod_{i=1}^n \frac{1}{(1-{\bold t}^{a_i})}$.
\end{definition}

Notice that $\varphi_A$ vanishes outside of $\pos(A):=\{\sum\lambda_ia_i\in\RR^n | \lambda_i\geq 0, 1\leq i\leq n\}$.
The type of this function and how this function works is very important. Indeed the values of $\varphi_A$ come from a quasi-polynomial.
Blakley showed in \cite{B} that $\NN^d$ can be decomposed into a finite number of
parts, called chambers, in such a way that $\varphi_A$ is a quasi-polynomial of degree $n-d$ in each chamber.
Later, Sturmfels in \cite{St} investigated these decompositions and
the differences of polynomials from one piece to another.

Here we briefly introduce the basic facts and necessary terminology
of vector partition functions, specially the
chambers and the polynomials (quasi-polynomials) obtained from
vector partition functions corresponding to a matrix $A$.
For more details about the vector partition function, we refer the
reader to \cite{B, BV, St}.

\subsubsection{Polyhedral Sets}
\begin{definition}(See also \cite{Zie})
A polyhedral complex $\Im$ is a finite collection of polyhedra in $\RR^d$ such that
\begin{enumerate}
\item the empty polyhedron is in $\Im$,
\item if $P \in \Im$, then all the faces of $P$ are also in $\Im$,
\item the intersection $P \cap Q$ of two polyhedra $P,Q \in \Im $ is a face of both of $P$ and of $Q$.
\end{enumerate}
The support $|\Im|$ of a polyhedral complex $\Im$ is the union of the polyhedra in $\Im$.
\end{definition}

\begin{definition}
If $\Im_1$ and $\Im_{2}$ are polyhedral complexes such that the support of  $\Im_2$ contains  the support of $\Im_1$ and such that for every $P \in \Im_1$ there is a $Q \in \Im_2$ with $Q \subset P$, then $\Im_2$ is a refinement of $\Im_1$.The common refinement of a family of polyhedra is a polyhedral complex such that its support is the
union of the polyhedra and such that each member is the intersection of some of the polyhedra.
\end{definition}

\subsubsection{ Chamber complex of a vector partition function.\\\\}
Let $\sigma\subseteq \{1, \ldots, n\}$.
We will say that $\sigma$ is independent if for all
 $i\in \sigma$, the vectors $a_i$ are
 linearly independent.
If $\sigma$ is independent, we set $A_\sigma:=(a_{i})_{i\in\sigma}$ and denote by $\Lambda_\sigma$
the $\ZZ$-module generated by the columns of $A_\sigma$ as its base and $\partial Pos(A_{\sigma})$ the boundary of $ \pos(A_{\sigma})$. When $\sigma$
has $d$ elements ({\it i. e.} is a maximal independent set),  $\Lambda_\sigma$
is a sublattice of $\ZZ^d$.

Let $\sum_{A}$ be the set of all simplicial cones whose extremal rays  are generated by $d$-linearly independent column vectors of $A$.
\begin{definition}
The chamber complex of a vector partition function as defined is the common refinement of all simplicial cones $Pos(A_{\sigma})$. The maximal chambers $C$ of the chamber complex of $A$ are the connected components of
$Pos (A) - \bigcup_{ \ell \in \sum_{A}} \partial \ell$. These chambers are open and convex.
\end{definition}

Associated to each maximal chamber $C$ there is an index set
 $$\Delta(C):= \{\sigma\subset \{1, \ldots, n\}\; |\; C\subseteq \pos(A_\sigma)\}.$$ $\sigma\in \Delta(C)$
  is  called non-trivial if $G_\sigma:=\ZZ^d/\Lambda_\sigma\neq0$,
 equivalently if $\det (\Lambda_\sigma )\not= \pm 1$ ($G_\sigma$ is finite because $C\subseteq \pos(A_\sigma)).$

 \begin{example}\label{ex4vect}
Let $\varphi$ be the vector partition function corresponding to the matrix $A = \left(\begin{array}{cccc}2 & 3 & 6 & 7 \\1 & 1 & 1 & 1\end{array}\right)$ i.e., $a_1=(2,1)$ , $a_2= (3,1)$ , $a_3=(6,1)$ , $a_4 = (7,1)$.
Figure~\ref{figure:chamber.decomposition} shows the chamber decomposition associated to the vector partition function $\varphi$ and since  any arbitrary vectors $a_{i} ,a_{i+1}$ make an independent set,  the common refinement consists of disjoint union of open convex  polyhedral cones generated by $a_{i} ,a_{i+1}$. Figure~\ref{figure:chamber.complex} shows the common refinement of all simplicial cones $A_\sigma$. \\

  \end{example}
  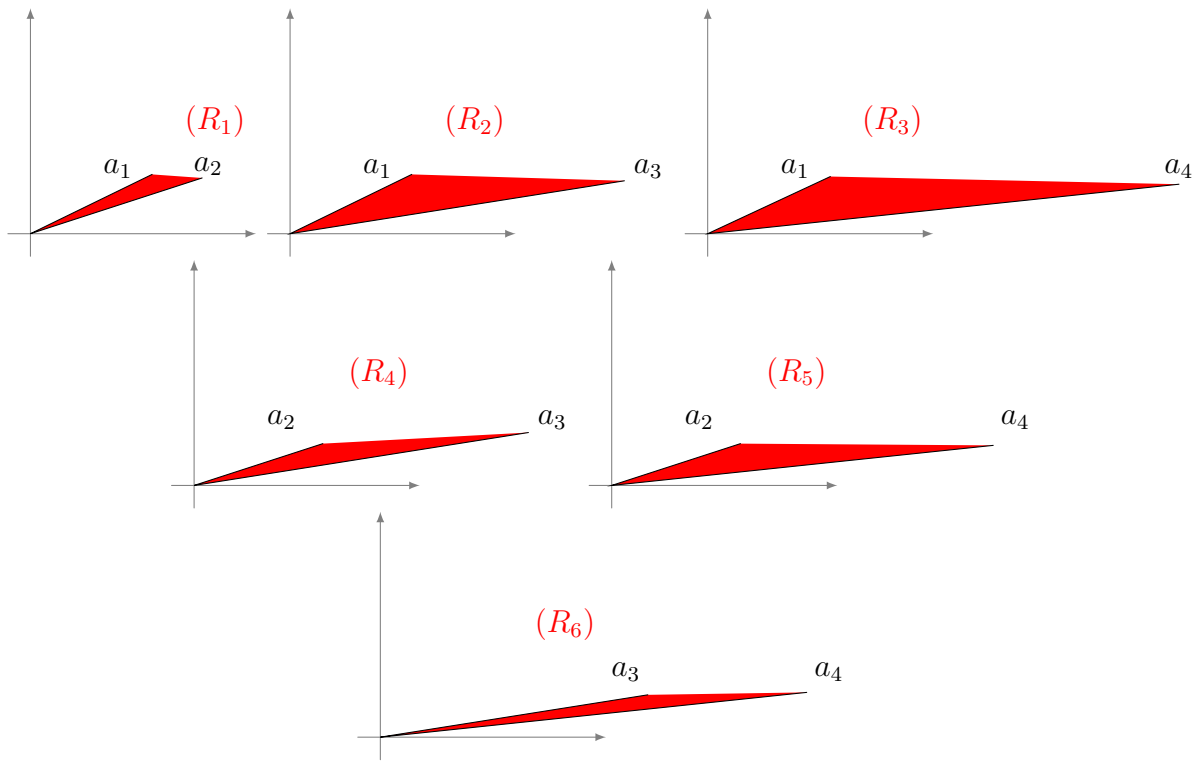
\begin{figure}
 \begin{tikzpicture}[scale=3,cap=round,>=latex]
      \newdimen\P
     \newdimen\Q
  \P=0.8cm
    \Q=0.6cm
  \coordinate (Origin)   at (0,0);
    \coordinate (XAxisMin) at (-0.1,0);
    \coordinate (XAxisMax) at (1,0);
    \coordinate (YAxisMin) at (0,-0.1);
    \coordinate (YAxisMax) at (0,1);
    \draw [thin, gray,-latex] (XAxisMin) -- (XAxisMax);
    \draw [thin, gray,-latex] (YAxisMin) -- (YAxisMax);
 \draw {  (26:\Q) -- (0,0)  -- (18:\P)  }[fill=red];
    \draw[red]                   (1,0.5) node[left]{$(R_1)$};
\draw[]                   (0.5,0.29) node[left] {$a_1$};
\draw[]                   (0.9,0.3) node[left] {$a_2$};

  \end{tikzpicture}
   \begin{tikzpicture}[scale=3,cap=round,>=latex]
    \coordinate (Origin)   at (0,0);
    \coordinate (XAxisMin) at (-0.1,0);
    \coordinate (XAxisMax) at (1,0);
    \coordinate (YAxisMin) at (0,-0.1);
    \coordinate (YAxisMax) at (0,1);
    \draw [thin, gray,-latex] (XAxisMin) -- (XAxisMax);
    \draw [thin, gray,-latex] (YAxisMin) -- (YAxisMax);
     \newdimen\P
     \newdimen\Q
  \P=0.6cm
    \Q=1.5cm
  \coordinate (center) at (0,0);
    \draw {(26:\P) -- (0,0)-- (9:\Q) }[fill=red];
     \draw[red]                   (1,0.5) node[left]{$(R_2)$};
\draw[]                   (0.5,0.29) node[left] {$a_1$};
\draw[]                   (1.7,0.29) node[left] {$a_3$};

  \end{tikzpicture}
\begin{tikzpicture}[scale=3,cap=round,>=latex]
 \coordinate (Origin)   at (0,0);
    \coordinate (XAxisMin) at (-0.1,0);
    \coordinate (XAxisMax) at (1,0);
    \coordinate (YAxisMin) at (0,-0.1);
    \coordinate (YAxisMax) at (0,1);
    \draw [thin, gray,-latex] (XAxisMin) -- (XAxisMax);
    \draw [thin, gray,-latex] (YAxisMin) -- (YAxisMax);
     \newdimen\P
     \newdimen\Q
  \P=0.6cm
    \Q=2.1cm
  \coordinate (center) at (0,0);
    \draw { (25:\P) -- (0,0) -- (6:\Q) }[fill=red];
     \draw[red]                   (1,0.5) node[left]{$(R_3)$};
\draw[]                   (0.5,0.29) node[left] {$a_1$};
\draw[]                   (2.2,0.29) node[left] {$a_4$};

  \end{tikzpicture}

\begin{tikzpicture}[scale=3,cap=round,>=latex]
 \coordinate (Origin)   at (0,0);
    \coordinate (XAxisMin) at (-0.1,0);
    \coordinate (XAxisMax) at (1,0);
    \coordinate (YAxisMin) at (0,-0.1);
    \coordinate (YAxisMax) at (0,1);
    \draw [thin, gray,-latex] (XAxisMin) -- (XAxisMax);
    \draw [thin, gray,-latex] (YAxisMin) -- (YAxisMax);
     \newdimen\P
     \newdimen\Q
  \P=0.6cm
    \Q=1.5cm
  \coordinate (center) at (0,0);
    \draw { (18:\P) -- (0,0) -- (9:\Q) }[fill=red];
         \draw[red]                   (1,0.5) node[left]{$(R_4)$};
\draw[]                   (0.5,0.29) node[left] {$a_2$};
\draw[]                   (1.7,0.29) node[left] {$a_3$};
  \end{tikzpicture}
  \begin{tikzpicture}[scale=3,cap=round,>=latex]
 \coordinate (Origin)   at (0,0);
    \coordinate (XAxisMin) at (-0.1,0);
    \coordinate (XAxisMax) at (1,0);
    \coordinate (YAxisMin) at (0,-0.1);
    \coordinate (YAxisMax) at (0,1);
    \draw [thin, gray,-latex] (XAxisMin) -- (XAxisMax);
    \draw [thin, gray,-latex] (YAxisMin) -- (YAxisMax);
     \newdimen\P
     \newdimen\Q
  \P=0.6cm
    \Q=1.7cm
  \coordinate (center) at (0,0);
    \draw { (18:\P) -- (0,0) -- (6:\Q) }[fill=red];
             \draw[red]                   (1,0.5) node[left]{$(R_5)$};
\draw[]                   (0.5,0.29) node[left] {$a_2$};
\draw[]                   (1.9,0.29) node[left] {$a_4$};
  \end{tikzpicture}

  \begin{tikzpicture}[scale=3,cap=round,>=latex]
 \coordinate (Origin)   at (0,0);
    \coordinate (XAxisMin) at (-0.1,0);
    \coordinate (XAxisMax) at (1,0);
    \coordinate (YAxisMin) at (0,-0.1);
    \coordinate (YAxisMax) at (0,1);
    \draw [thin, gray,-latex] (XAxisMin) -- (XAxisMax);
    \draw [thin, gray,-latex] (YAxisMin) -- (YAxisMax);
     \newdimen\P
     \newdimen\Q
  \P=1.2cm
    \Q=1.9cm
  \coordinate (center) at (0,0);
    \draw { (9:\P) -- (0,0) -- (6:\Q) }[fill=red];
             \draw[red]                   (1,0.5) node[left]{$(R_6)$};
\draw[]                   (1.2,0.29) node[left] {$a_3$};
\draw[]                   (2.1,0.29) node[left] {$a_4$};
  \end{tikzpicture}
  \caption{The chamber decomposition of Example \ref{ex4vect}.}
    \label{figure:chamber.decomposition}
   \end{figure}

  \begin{figure}[ht]
  \begin{center}

  \begin{tikzpicture}[scale=4,cap=round,>=latex]
    \coordinate (Origin)   at (8,9);
    \coordinate (XAxisMin) at (-0.1,0);
    \coordinate (XAxisMax) at (2,0);
    \coordinate (YAxisMin) at (0,-0.1);
    \coordinate (YAxisMax) at (0,1);
    \draw [thin, gray,-latex] (XAxisMin) -- (XAxisMax);
    \draw [thin, gray,-latex] (YAxisMin) -- (YAxisMax);
     \newdimen\P
     \newdimen\Q
  \P=0.6cm
    \Q=1.5cm
  \coordinate (center) at (0,0);
    \draw {(26:\P) -- (0,0)-- (18:\Q) }[fill=red];
    \draw {(18:\P) -- (0,0)-- (9:\Q) }[fill=blue];
    \draw {(9:\P) -- (0,0)-- (5:\Q) }[fill=green];
\draw[]                   (0.5,0.29) node[left] {$a_1$};
\draw[]                   (1.0,0.27) node[left] {$a_2$};
\draw[]                   (1.7,0.26) node[left] {$a_3$};
\draw[]                   (1.8,0.18) node[left] {$a_4$};
  \end{tikzpicture}
 \caption{The chambers complex of Example \ref{ex4vect}.}
 \label{figure:chamber.complex}

\end{center}
 \end{figure}
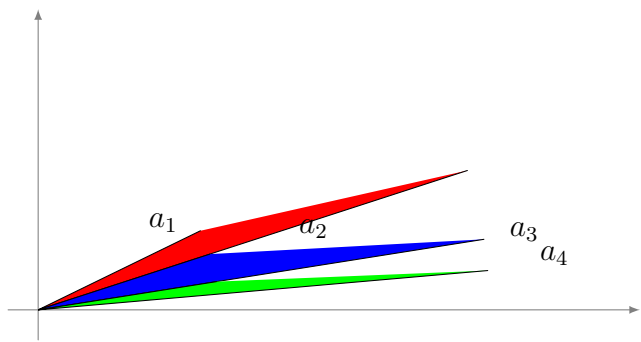
  If $A$ is an arbitrary matrix, it is not easy to explain precisely its corresponding chamber complex.
Here, in the following lemma we describe the chamber complex associated to a  special kind of $2\times n$-matrix that its columns are $(d_i, 1)$, $d_i\in \NN$ which it plays an important role in the proof of proposition \ref{hilbert bigraded}.
\begin{lemma}\label{chamber of n vector}
Let $$A=\left(\begin{array}{lll}
d_1&\ldots&d_n\\
1&\ldots&1\end{array}\right)$$ be a $2\times n$-matrix with entries in $\NN$
such that $d_1\leq\ldots \leq d_n$. Then the chambers corresponding to
 $\pos(A)$ are positive polyhedral cones $\Delta$ where
$\Delta$ is generated by $\{(d_i, 1), (d_{i+1}, 1)\}$ for all
$d_i\neq d_{i+1}$ and $i$ runs over $\{1, \ldots, n\}$.
\end{lemma}



 Now, we are ready to state the vector partition function theorem, which
relies on the chamber decomposition of $\pos(A)\subseteq \NN^d$.

 \begin{theorem}\label{vector partition}(See \cite[Theorem 1]{St})
 For each chamber $C$ of maximal dimension in the chamber complex corresponding to matrix $A$,
 there exist a polynomial $P$ of degree
 $n-d$, a collection of polynomials  $Q_\sigma$
 and  functions $\Omega_\sigma: G_\sigma\setminus\{0\}\rightarrow \QQ$ indexed by non-trivial $\sigma\in \Delta(C)$
such that, if $u\in \NN A\cap \overline{C}$,
 $$\varphi_A(u)=P(u)+\sum\{\Omega_\sigma([u]_\sigma).Q_\sigma(u) : \sigma\in \Delta(C) , [u]_\sigma\neq 0\}$$
where $[u]_\sigma$ denotes the image of $u$ in $G_\sigma$. Furthermore, $\deg (Q_\sigma )=\#\sigma-d$.\\
\end{theorem}

\begin{corollary}\label{vector partition2}
 For each chamber $C$ of maximal dimension in the chamber complex of $A$,
 there exists a collection of polynomials  $Q_\tau$ for $\tau\in \ZZ^d/\Lambda$ such that
 $$
 \varphi_A(u)=Q_\tau (u), \ \hbox{if}\ u\in \NN A\cap \overline{C}\ \hbox{and}\ u \in \tau+\Lambda_{C}
 $$
 where $\Lambda_{C} =  \cap_{\sigma \in \Delta(C)} \Lambda_{\sigma}.$
\end{corollary}
\begin{proof}
The class $\tau$ of $u$ modulo $\Lambda$ determines  $[u]_\sigma$  in
$G_\sigma=\ZZ^d/\Lambda_\sigma$. The term of the right-hand side of the equation in the above theorem is
a polynomial determined by $[u]_\sigma$, hence by $\tau$.
\end{proof}

Notice that setting $\Lambda$ for the intersection of the lattices
$\Lambda_\sigma$ with $\sigma$ maximal, the class of $u$ mod $\Lambda$  determines the class of $u$ mod $\Lambda_{C}$. Thus, the corollary holds with $\Lambda$ in place of $\Lambda_{C}$.\\

In the next section, we will use this theorem to show that
the Hilbert function of a graded ring is determined by a finite collection of polynomials.
\section{Hilbert functions of non-standard bigraded rings}

Let  $S=k[y_1, \ldots, y_m]$ be a $\ZZ^{d-1}$-graded polynomial ring over a field and let
$I=(f_1,\ldots ,f_n)$ be a graded ideal
such that its generators $f_i$ are homogeneous of degree $d_i$.
To get information about the behavior  of  $i$-syzygy module of $I^t$ as $t$ varies, we
 pass to Rees algebra $\R_I=\oplus_{t\geq 0}I^t$.
 $\R_I$ has a $(\ZZ^{d-1}\times\ZZ )$-graded
 algebra structure
 such that $\left(\R_I\right)_{(\mu, n)}=(I^n)_\mu$.

 Recall that $\R_I$ is a graded
 quotient of $R:=S[x_1, \ldots, x_n]$ with grading extended from the one of $S$ by setting $\deg (a):=(\deg (a), 0)$ for
 $a\in S$ and $\deg (x_j):=(d_j ,1)$ for all $j= 1, \ldots, n$.
 As noticed in \cite{BCH}, if  $\G_\bullet$ is a $\ZZ^d$-graded free $R$-resolution of  $\R_I$,
 then by setting $$B:=k[x_1, \ldots, x_n]=R/(y_1, \ldots, y_m)$$
 we will have the following equality:
 $$
 \tor_i^S(I^t, k)_\mu=H_i(\G_\bullet\otimes_R B)_{(\mu,t)}.
 $$

The complex $\G_\bullet\otimes_R B$ is a $\ZZ^d$-graded
 complex of free $S$-modules. Its homology modules are therefore
 finitely generated $\ZZ^d$-graded $S$-modules, on which we will apply results derived from
 the ones on vector partition functions describing the Hilbert series of $S$.



Now we are ready to prove the main result of this section.
\begin{proposition}\label{hilbert bigraded}
Let $B=k[T_1, \ldots, T_n]$ be a bigraded polynomial ring over  field $k$
with $\deg(T_i)=(d_i, 1)$. Assume that the number of distinct $d_i$'s is $r\geq 2$.
Then there exists a finite index sublattice $\Lambda$ of $\ZZ^2$
and collections of polynomials
$Q_{ij}$ of degree $n-2 $ for $1\leq i\leq r-1$ and $1\leq j\leq s$ such that
for any $(\mu, \nu)\in \ZZ^2\cap R_i$ and $\overline{\left(\mu, \nu\right)}\equiv g_j \mod \Lambda$
in ${\ZZ^2}/{\Lambda}:=\{g_1, \ldots, g_s\}$,
$$HF(B, (\mu, \nu))= Q_{ij}(\mu, \nu)$$
where
$R_i$ is the convex polyhedral cone generated by linearly independent vectors $\{(d_i, 1), (d_{i+1}, 1)\}$.

Furthermore, $Q_{ij}(\mu, \nu) = Q_{ij}(\mu^{\prime}, \nu^{\prime}) $ if $ \mu - \nu d_{i} \equiv  \mu^{\prime}-\nu^{\prime}d_i  \mod (\det (\Lambda))$.
\end{proposition}
\begin{proof}
Let
$$A=\left(\begin{array}{lll}
d_1 & \ldots & d_n\\
1 & \ldots & 1\end{array}\right)$$
be a $2\times n$-matrix corresponding to degrees of $T_i$.

The Hilbert function in degree
${\bf u}=(\mu, \nu)$ is the number of monomials $T_1^{\alpha_1}\ldots T_n^{\alpha_n}$
such that $\alpha_1(d_1, 1)+\ldots+\alpha_n(d_n, 1)=(\mu, \nu)$.
This equation is equivalent to the system of linear equations
$$A.\left(\begin{array}{c}\alpha_1\\
\vdots\\
\alpha_n\end{array}\right)=(\begin{array}{ll}
\mu & \nu\end{array}).$$
In this sense $HF(B, (\mu, \nu))$ will be the value of vector partition function
at  $(\mu, \nu)$. Assume that $(\mu, \nu)$ belongs to the chamber $R_i$ which is the convex
polyhedral cone generated by $\{(d_i, 1), (d_{i+1}, 1)\}$.
By \ref{vector partition2}, we know that for $(\mu, \nu)\in R_i$ and $(\mu, \nu)\equiv g_j\mod(\det \Lambda)$,
\begin{equation}\label{1}
\varphi_A(\mu, \nu)=Q_{ij}(\mu, \nu).
\end{equation}

\end{proof}
Notice that in Proposition \ref{hilbert bigraded}, if moreover we suppose that $d_i\neq d_j$ for all $i\neq j$,
then the Hilbert function in degree $(\mu, \nu)$ will also be the number of lattice paths from $(0,0)$
to $(\mu, \nu)$ in the lattice $S = \{(d_i, 1) | 1 \leqslant i \leqslant n \}$, but this can fail if some of the $d_i$s are equal.
For example if one has $d_i=d_{i+1}<d_{i+2}$, so the independent sets of vectors
$\{(d_i, 1), (d_{i+2}, 1)\}$ and $\{(d_{i+1}, 1), (d_{i+2}, 1)\}$
generate the same chamber and the number of lattice paths from
$(0, 0)$ to $(\mu, \nu)$ via lattice points is less than $\HF(B, (\mu, \nu))$.

Notice that although in this section the Hilbert function of special bigraded polynomial rings are obtained, this result can be given for every arbitrary $\ZZ^d$-graded polynomial rings using the concept of vector partition function theorem.
\section{Betti numbers of powers of ideals }


We now turn to the main result on Betti numbers of powers of ideals. Without additional effort, we  treat
the following more general situation of a collection of graded ideals and include a graded module $M$. Hence we will study the behaviour of $\dim_k \tor_i^R(MI_1^{t_1}\cdots I_{s}^{t_s},k)_{\mu}$ for $\mu\in \ZZ^p$ and ${\bold t}\gg 0$. To this aim we first use the
important fact that the module
$$
B_i:=\oplus_{t_1,\ldots ,t_s}\tor_i^R(MI_1^{t_1}\cdots I_{s}^{t_s},k)
$$
is a finitely generated $(\ZZ^p\times \ZZ^s)$-graded ring, over $k[T_{i,j}]$ setting $\deg (T_{i,j})=(\deg (f_{i,j}),e_i)$
with $e_i$ the $i$-th canonical generator of $\ZZ^s$ and for fixed $i$ the $f_{i,j}$ form a set of minimal generators of  $I_i$.
Hence $\tor_i^R(MI_1^{t_1}\cdots I_{s}^{t_s},k)_\mu =(B_i)_{\mu ,t_1e_1+\cdots +t_se_s}$.

The following result applied to $B_i$ will then give the asymptotic behaviour of Betti numbers. In particular case
if we have  a single $\ZZ$-graded ideal, we will use it to give a simple description of this eventual behaviour.

\subsection{The general case }

Let $\varphi : \ZZ^n \rightarrow \ZZ^d $ with $\varphi(\NN^n) \subseteq  \NN^d$ be a positive $\ZZ^d$-grading of $R:=k[T_{i,j}]$. Set $\ZZ^n:=\sum_{i=1}^n\ZZ e_i$, let $E$ be the set of $d$-tuples $e=(e_{i_1},...,e_{i_d})$ such that $(\varphi(e_{i_1}),...,\varphi(e_{i_d}))$ generates a lattice $\Lambda_e$ in $\ZZ^d$, and set
$$
\Lambda :=\cap_{e\in E} \Lambda_e,\quad\quad \xymatrix{s_\Lambda : \ZZ^d \ar^{can}[r]&\ZZ^d/\Lambda\\}.
$$

Denote by $C_i$, $i\in F$, the maximal cells in the  chamber complex associated to $\varphi$.
One has $$\overline{C_i} = \{ \xi\ \vert\ H_{i,j}(\xi ) \geqslant 0, \ 1\leqslant j \leqslant d \}$$ where $H_{i,j}$ is a linear form
in $\xi \in \ZZ^d$.

Let  $S=k[y_1, \ldots, y_m]$ be a $\ZZ^{p}$-graded polynomial ring over a field. Assume
that $\deg (y_j)\in \NN^p$ for any $j$, and let
$I_i=(f_{i,1},\ldots ,f_{i,r_i})$ be ideals, with $f_{i,j}$ homogeneous of degree $d_{i,j}$.

Consider $R:=k[T_{i,j}]_{1\leq i\leq s,\,1\leq j\leq r_i}$, set $\deg (T_{i,j})=(\deg (f_{i,j}),e_i)$, with $e_i$ the $i$-th canonical generator of $\ZZ^s$
and  the induced grading $\varphi : \ZZ^{r_1+\cdots +r_s} \rightarrow \ZZ^d :=\ZZ^p\times \ZZ^s$ of $R$.

Denote as above by $\Lambda$ the lattice in $\ZZ^d$ associated to $\varphi$, by $s_\Lambda : \ZZ^d \to \ZZ^d/\Lambda$ the  canonical morphism and by
$C_i$, for $i\in F$, the maximal cells in the the chamber complex associated to $\varphi$.
One has $\overline{C_i} = \{ (\mu ,t)\ \vert\ H_{i,j}(\mu ,t ) \geqslant 0, 1\leqslant j \leqslant d \}$ where $H_{i,j}$ is a linear form
in $(\mu ,t) \in \ZZ^p\times \ZZ^s=\ZZ^d$.

\begin{proposition}\label{general case}
With notations as above, let $B$ be a finitely generated $\ZZ^d$-graded $R$-module. There exist a finite set $U$ and convex sets of
dimension $d$ in $\RR^d$ of the form
$$
\Delta_u = \{ x \ \vert\ H_{i,j}(x) \geqslant a_{u,i,j}, \forall (i,j)\in G_u \} \subseteq \RR^d
$$
for $u \in U$, with $a_{u,i,j}=H_{i,j}(a)$ for $a\in \cup_{\ell}\supp_{\ZZ^d}(\tor^R_\ell (B,k))$, $G_u\subset F\times \{ 1,\ldots ,d\}$  and polynomials $P_{u,\tau}$ for $u\in U$ and $\tau\in \ZZ^d/\Lambda$ such that for all $\xi$ belonging to $\Delta_u$ we have
$$
\dim_k (B_{\xi }) = P_{u,s_{\Lambda} (\xi )} (\xi )
$$
and
for all $\xi$ out of $\bigcup_{u\in U}\Delta_u$ we have
$$\dim_k (B_{\xi }) = 0.$$
\end{proposition}
\begin{proof}
By Proposition \ref{Hilbert nonstandard}, there exists a polynomial $\kappa_B (t_1,\ldots ,t_d)$ with integral coefficients such that
$$
H(B;{\bold t}) =\kappa_B ({\bold t})H(R;{\bold t})
$$
and $\kappa_B({\bold t})=\sum_{a\in A} c_a{\bold t}^a$ with $A\subset \cup_{\ell}\supp_{\ZZ^d}(\tor^R_\ell (B,k))$. Let $D_i:=\cup_j \{ x\ \vert\ H_{i,j}(x)=0\}$ be the minimal union of hyperplanes containing the border of $C_i$. The union $C$ of the
convex sets $\overline{C_i}+a$ can be decomposed into a finite union of convex sets $\Delta_u$, each $u\in U$
corresponding to one connected component of $C\setminus\bigcup_{i,a}(D_i+a)$. (Notice that
$\RR^d \setminus \bigcup_{i,a}(D_i+a)$ has finitely many connected components, which are convex sets of
the form of $\Delta_u$, and that we may drop the ones not contained in $C$ as
the dimension of $B_\xi$ is zero for $\xi$ not contained in any $\overline{C_i}+a$.)
We define $u$ as the set of pairs $(i,a)$ such that $(C_i+a)\bigcap \Delta_u\neq \emptyset$, and remark
that if $(i,a)\in u$ then $(j,a)\not\in u$ for $j\not= i$.

If $\xi\not\in \bigcup_i \overline{C_i}+a$, then $\dim_k R_{\xi -a}=0$, while if $\xi\in\overline{C_i}+a$ then
it follows from Corollary \ref{vector partition2} that there exist polynomials $Q_{i,\tau}$ such that
$$
\dim_k R_{\xi -a}=Q_{i,\tau }(\xi -a)\quad \hbox{if}\ \xi -a\equiv\tau\mod (\Lambda).
$$
Hence, setting $Q'_{i,a,\tau }(\xi ):=Q_{i,\tau +a}(\xi -a)$, one gets the conclusion
with
$$
P_{u,\tau }=\sum_{(i,a)\in u} c_aQ'_{i,a,\tau}.
$$
\end{proof}

\begin{remark}\label{severalmodules}
The above proof shows that if one has a finite collection of modules $B_i$,
setting  $A:=\cup_{i,\ell}\supp_{\ZZ^d}\tor_\ell^R(B_i,k)$, there exist
convex polyhedral cones $\Delta_u$ as above on which any $B_i$ has its Hilbert function given by a
quasi-polynomial
with respect to the lattice $\Lambda$.
\end{remark}
We now turn to the main result of this article. The more simple, but important, case of powers of
an ideal in a positively $\ZZ$-graded algebra over a field will be detailed just after.

\begin{theorem}\label{tor-general case}
In the situation above, there exist a finite set $U$ and a finite number of polyhedral convex cones
$$
\Delta_u = \{(\mu,t) | H_{i,j}(\mu,t)\geqslant a_{u,i,j}, (i,j)\in G_u \} \subseteq \RR^d,
$$
polynomials $P_{\ell ,u,\tau}$ for $u\in U$ and $\tau\in \ZZ^d/\Lambda$ such that, for any $\ell$,
$$
\dim_k (\tor_\ell^S(MI_1^{t_1}...I_s^{t_s}, k)_{\mu }) = P_{\ell,u,s_{\Lambda}} (\mu,t),\quad \forall (\mu,t)\in \Delta_u,
$$
and $\dim_k (\tor_\ell^S(MI_1^{t_1}...I_s^{t_s}, k)_{\mu}) = 0$ if $(\mu,t)\not\in \cup_{u\in U}\Delta_u$.

 Furthermore, for any $(u,i,j)$, $a_{u,i,j}=H_{i,j}(b)$, for some
 $$
 b\in
\bigcup_{i,\ell}\supp_{\ZZ^d}\tor_\ell^R(\tor^S_i(M\R_{I_1,\ldots ,I_s},R),k).
$$
\end{theorem}

\begin{proof}
We know from  \cite{BCH} that $B_i:=\oplus_{t_1,\ldots ,t_t}\tor_i^S(MI_1^{t_1}\cdots I_{s}^{t_s},k)$
is a finitely generated $\ZZ^d$-graded module over $R$. As $B_i\not= 0$ for only finitely many $i$,
the conclusion follows from Proposition \ref{general case} and Remark \ref{severalmodules}.
\end {proof}

The above results tell us that $\RR^d$ can be decomposed in a finite union of convex polyhedral cones
$\Delta_u$ on which, for any $\ell$, the dimension of  $\tor_\ell^S(MI_1^{t_1}...I_s^{t_s}, k)_{\mu}$, as a function of
$(\mu ,t)\in \ZZ^{p+s}$ is a quasi-polynomial with respect to a lattice determined by the degrees of the generators of
the ideals $I_1,\ldots ,I_s$.

This general finiteness statement may lead to pretty complex decompositions
in general, that depends on the number of ideals and on arithmetic properties of the sets of degrees of generators.
This complexity is reflected both by the covolume of $\Lambda$ as defined above and by the number of simplicial chambers
in the chamber complex associated to $\varphi$.
\subsection{The case of one graded ideal on a positively $\ZZ$-graded ring}
In the case that the polynomial ring is multigraded and we have finitely many ideals, the results are done in the Theorem \ref{tor-general case}.
To make the result more clear and maybe more useful, let us explain the details of this theorem in an important and simple case when we have
 one ideal in a positively $\ZZ$-graded polynomial ring over a field.

We begin by an elementary lemma.

\begin{lemma}\label{intersection}
For a strictly increasing sequence $d_1<\cdots <d_r$, and points of coordinates $(\beta_1^j, \beta_2^j)\in \RR^2$ for $1\leq j\leq N$, consider the half-lines $$L_i^j:= \{ (\beta_1^j, \beta_2^j)+\lambda (d_i, 1),\ \lambda\in \RR_{\geq 0} \}$$ and set $L_i^j(t):=L_i^j\cap \{ y=t\}$.
Then  there exist a positive integer $t_0$ and permutations $\sigma_i$, for $i=1,\cdots, r$ in the permutation group $S_N$
such that for all $t\geq t_0$,  the following properties are satisfied:

\item (1)
$L_{i}^{\sigma_i (1)}(t) \leqslant L_{i}^{\sigma_i (2)}(t) \leqslant \dots \leqslant L_{i}^{\sigma_i (N)}(t)$ for $1 \leqslant i\leqslant r$,
\item (2)
$L_{i}^{\sigma_i (N)} (t)\leqslant L_{i+1}^{\sigma_{i+1} (1)}(t)$.

Moreover $t_0$ can be taken as the biggest second coordinate of the intersection points of all pairs of half-lines(See Figure ~\ref{figure:shifts} and Figure~\ref{figure:t-large}).
\end{lemma}
\begin{proof}
If two half-lines $L_i^j$ and $L_u^v$ intersect at a unique point
 $A(x_A, y_A)$, then \\\\
 $$
 y_A=\frac{\det\left(\begin{array}{cc}\beta_1^v&d_u\\ \beta_2^v& 1\end{array}
 \right)-\det\left(\begin{array}{cc}\beta_1^j& d_i\\
 \beta_2^j& 1\end{array}\right)}{d_i-d_u}.
 $$
 Choose $t_0$ as the max of $y_A$, $A$ running over the intersection points. Only for $t\in [ t_0,+\infty [$ the ordering of the intersection points $L_i^j(t)$ on the line $\{ y=t\}$ is independent of $t$.
 Furthermore, as the $d_i$'s are strictly increasing (2) holds, which shows (1) as the ordering is independent of $t$.\\\\
 \end{proof}

\begin{figure}[ht]
  \centering
  \begin{tikzpicture}
    \coordinate (Origin)   at (0,0);
    \coordinate (XAxisMin) at (-1,0);
    \coordinate (XAxisMax) at (12,0);
    \coordinate (YAxisMin) at (0,-1);
    \coordinate (YAxisMax) at (0,5);
    \draw [thin, gray,-latex] (XAxisMin) -- (XAxisMax);
    \draw [thin, gray,-latex] (YAxisMin) -- (YAxisMax);

    \clip (-3,-1) rectangle (95cm,10cm); 
    \coordinate (Aone) at (1,2);
    \coordinate (Atwo) at (2,2);
    \coordinate (Athree) at (3,2);
     \coordinate (Bone) at (4,4);
    \coordinate (Btwo) at (5,4);
    \coordinate (Bthree) at (6,4);

    \coordinate (Done) at (3,5);
    \coordinate (Dtwo) at (4,5);
    \coordinate (Dthree) at (5,5);

    \foreach \x in {-7,-6,...,7}{
      \foreach \y in {-7,-6,...,7}{
      }
    }
    \draw [ultra thick,-latex,red] (0,1)
        -- (Aone) node [above left] {};
    \draw [ultra thick,-latex,red] (0,1)
        -- (Atwo) node [below right] {};
     \draw [ultra thick,-latex,red] (0,1)
        -- (Athree) node [above left] {};

\draw[red]                   (0,1) node[left] {} -- (7,8)
                                        node[right]{$L_{1}^{3}$};

\draw[red]                   (0,1) node[left] {$(\beta_1^3, \beta_2^3)$} -- (12,7)
                                        node[right]{$L_{2}^{3}$};

\draw[red]                   (0,1) node[left] {} -- (14.1,5.7)
                                        node[right]{$L_{3}^{3}$};

    \draw [ultra thick,-latex,violet] (3,3)
        -- (Bone) node [above left] {$$};
    \draw [ultra thick,-latex,violet] (3,3)
        -- (Btwo) node [above left] {$$};
     \draw [ultra thick,-latex,violet] (3,3)
        -- (Bthree) node [above left] {$$};

\draw[violet]                   (3,3) node[left] {} -- (9,9)
                                        node[right]{$L_{1}^{2}$};

\draw[violet]                   (3,3) node[left] {$(\beta_1^2, \beta_2^2)$} -- (14,8.5)
                                        node[right]{$L_{2}^{2}$};

\draw[violet]                   (3,3) node[left] {} -- (14.1,6.7)
                                        node[right]{$L_{3}^{2}$};
    \draw [ultra thick,-latex,brown] (2,4)
        -- (Done) node [above left] {$$};
    \draw [ultra thick,-latex,brown] (2,4)
        -- (Dtwo) node [above left] {$$};
     \draw [ultra thick,-latex,brown] (2,4)
        -- (Dthree) node [above left] {$$};

\draw[brown]                   (2,4) node[left] {} -- (7,9)
                                        node[right]{$L_{1}^{1}$};

\draw[brown]                   (2,4) node[left] {$(\beta_1^1, \beta_2^1)$} -- (12,9)
                                        node[right]{$L_{2}^{1}$};

\draw[brown]                   (2,4) node[left] {} -- (14,8)
                                        node[right]{$L_{3}^{1}$};

  \end{tikzpicture}
  \caption{3-Shifts.}
  \label{figure:shifts}
\end{figure}
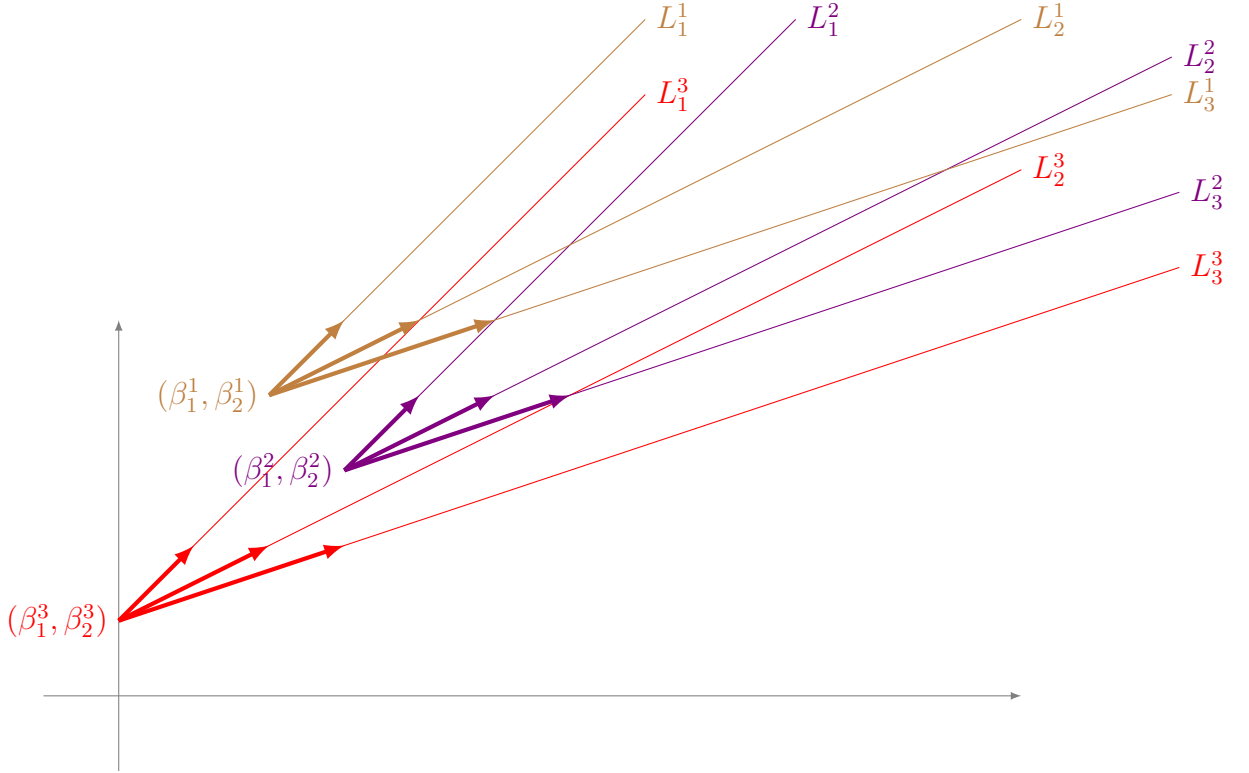

\begin{figure}[ht]
  \centering
  \begin{tikzpicture}
    \coordinate (Origin)   at (0,0);
    \coordinate (XAxisMin) at (-1,0);
    \coordinate (XAxisMax) at (12,0);
    \coordinate (YAxisMin) at (0,-1);
    \coordinate (YAxisMax) at (0,5);
    \draw [thin, gray,-latex] (XAxisMin) -- (XAxisMax);
    \draw [thin, gray,-latex] (YAxisMin) -- (YAxisMax);

    \clip (-3,-1) rectangle (95cm,10cm); 
    \coordinate (Aone) at (1,2);
    \coordinate (Atwo) at (2,2);
    \coordinate (Athree) at (3,2);
     \coordinate (Bone) at (4,4);
    \coordinate (Btwo) at (5,4);
    \coordinate (Bthree) at (6,4);

    \coordinate (Done) at (3,5);
    \coordinate (Dtwo) at (4,5);
    \coordinate (Dthree) at (5,5);

    \foreach \x in {-7,-6,...,7}{
      \foreach \y in {-7,-6,...,7}{
      }
    }
        -- (Aone) node [above left] {};
        -- (Atwo) node [below right] {};
        -- (Athree) node [above left] {};

\draw[red]                   (1,4) node[left] {} -- (8,9)
                                        node[right]{$L_{1}^{3}$};

\draw[red]                   (5,4) node[left] {$$} -- (14,9)
                                        node[right]{$L_{2}^{3}$};

\draw[red]                   (9,4) node[left] {} -- (14.1,6)
                                        node[right]{$L_{3}^{3}$};

        -- (Bone) node [above left] {$$};
        -- (Btwo) node [above left] {$$};
        -- (Bthree) node [above left] {$$};

\draw[violet]                   (2,4) node[left] {} -- (9,9)
                                        node[right]{$L_{1}^{2}$};

\draw[violet]                   (4,4) node[left] {$$} -- (13,9)
                                        node[right]{$L_{2}^{2}$};

\draw[violet]                   (8,4) node[left] {} -- (14.1,6.5)
                                        node[right]{$L_{3}^{2}$};
        -- (Done) node [above left] {$$};
        -- (Dtwo) node [above left] {$$};
        -- (Dthree) node [above left] {$$};

\draw[brown]                   (0,4) node[left] {} -- (7,9)
                                        node[right]{$L_{1}^{1}$};

\draw[brown]                   (3,4) node[left] {$$} -- (12,9)
                                        node[right]{$L_{2}^{1}$};

\draw[brown]                   (7,4) node[left] {} -- (14,7)
                                        node[right]{$L_{3}^{1}$};

\draw[black]                   (0,4) node[left] {$n \gg 0$} -- (14,4)
                                        node[right]{$$};

  \end{tikzpicture}
  \caption{Regions when n is sufficiently large.}
  \label{figure:t-large}
\end{figure}
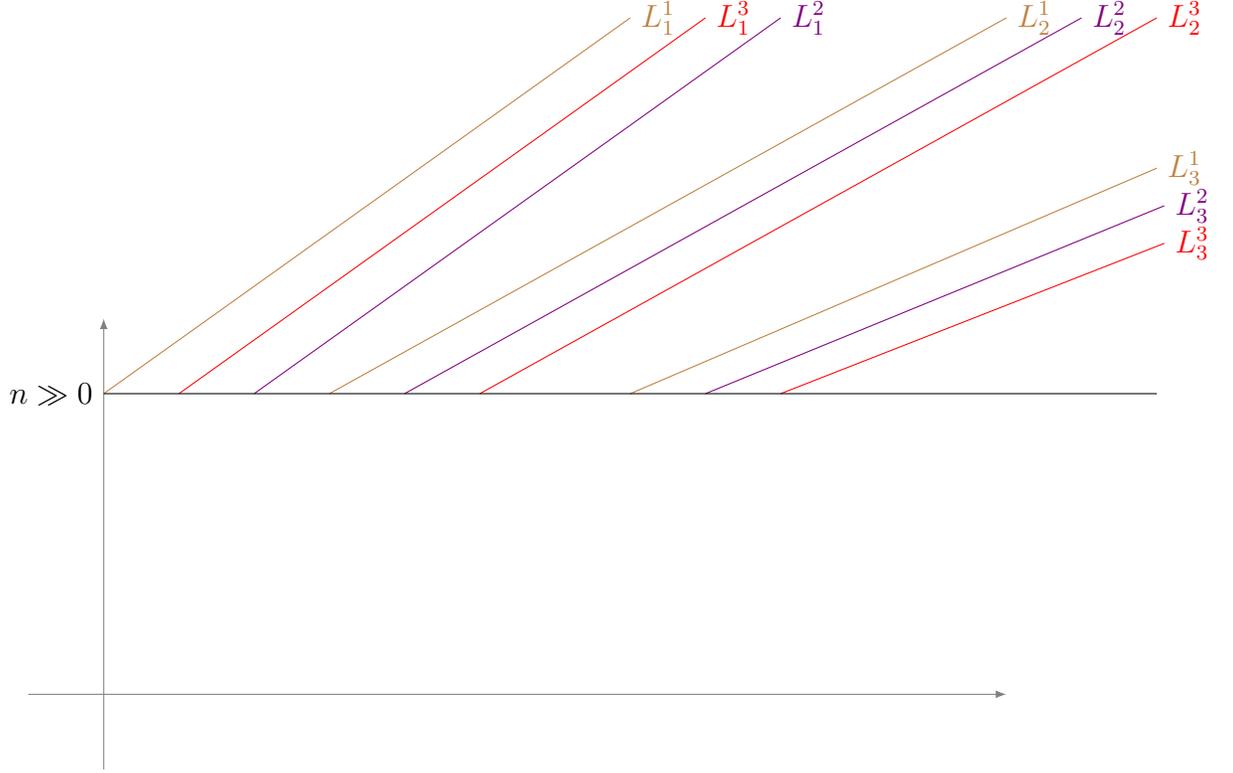


Now we are ready to give a specific description of $\tor^S_i(I^t, k)$ in the case of a $\ZZ$-graded ideal.
Let $E:=\{ e_1,\ldots ,e_s\}$ with $e_1<\cdots <e_s$ be a set of positive integers. For $\ell$ from
$1$ up to $s-1$, let $$\Omega_\ell :=\{ a{{e_\ell}\choose{1}}+b{{e_{\ell +1}}\choose{1}},\ (a,b)\in \RR_{\geq 0}^2\} $$ be the closed cone spanned by ${{e_\ell}\choose{1}}$ and ${{e_{\ell +1}}\choose{1}}$.  For integers $i\neq j$, let $\Lambda_{i,j}$ be the lattice spanned by ${{e_i}\choose{1}}$ and ${{e_{j}}\choose{1}}$ and
$$
\Lambda_\ell :=\cap_{i\leq \ell < j}\Lambda_{i,j}.
$$
Also we set
$$
\Lambda :=\cap_{i < j}\Lambda_{i,j} \hspace*{2mm} with \hspace*{2mm} \Delta = \det(\Lambda).
$$

In the case $E:=\{ d_1,\ldots ,d_r\}$, $e_1=d_1$ and $e_s=d_r$ and if $s\geq 2$, it follows from Theorem \ref{vector partition} that
\begin{itemize}
\item[(i)] $\dim_k B_{\mu ,t}=0$ if $(\mu ,t)\not\in \varOmega:= \bigcup_\ell \Omega_\ell$,

\item[(ii)] $\dim_k B_{\mu ,t}$ is a quasi-polynomial with respect to the lattice $\Lambda_\ell$ for $(\mu ,t)\in \Omega_\ell$.
\end{itemize}
Notice further that $\Lambda :=\cap_{i < j}\Lambda_{i,j}$ is a sublattice of $\Lambda_\ell$ for any $\ell$.

\begin{proposition} \label{z-graded}

In the above situation, if $M$ is a finitely
generated graded $B$-module, there exist $t_0$, $N$ and half-lines $L_i(t):=a_i t+b_i$ for $i=1,\ldots ,N$ with $b_i\in \ZZ$ and
$\{ a_1,\ldots ,a_N\}=E$ such that for  $t\geq t_0$:
\begin{itemize}
\item[(i)] $L_i (t)<L_j (t)\ \Leftrightarrow i<j$,
\item[(ii)] $M_{\mu ,t}=0$ if $\mu<L_1 (t)$ or $\mu >L_N (t)$,
\item[(iii)] For $t\geq t_0$ and $L_{i} (t)\leq \mu \leq L_{i+1}(t)$  $1\leq i < N $, $\dim_k (M_{\mu ,t})$ is a  quasi-polynomial $Q_{i}(\mu ,t)$
with respect to the lattice $\Lambda$.

\end{itemize}
\end{proposition}

\begin{proof}
By Proposition \ref{Hilbert nonstandard}, there exists a polynomial $P(x,y)$ with integral coefficients such that
$$
H(M;(x,y)) =P(x,y)H(B;(x,y))
$$
of the form $P(x,y)=\sum_{(a,b)\in A} c_{a,b}x^ay^b$ with $A\subset \cup_{\ell}\supp_{\ZZ^2}(\tor^R_\ell (M,k))$. Let
$$A=\{(\beta_1^{1}, \beta_2^{1}), \ldots, (\beta_1^{N}, \beta_2^{N})\}.$$
Now let $L_i^{j}(t):=d_i t+b_j$ be the half-line parallel to the vector $(d_i, 1)$ and
passing through the point $(\beta_1^j, \beta_2^j)$
for  $1\leq i, j\leq N$. Then item(i) follows directly from Lemma \ref{intersection} (i) and item(ii) from the fact that $M_{(\mu,t)}= 0$ unless $(\mu,t)\in \bigcup_{i=1}^{N} (\beta_1^{i}, \beta_2^{i})+\varOmega$.\\

To prove (iii), following \ref{intersection} we can consider two type of intervals as below:
$$
I_{i}^{j}:= [L_{i}^{\sigma_{i}(j)}(n),L_{i}^{\sigma_{i}(j+1)}(n)] \hspace*{2mm} for  \hspace*{2mm} j<N
$$
and
$$
I_{i}^{N}:= [ L_{i}^{\sigma_{i}(N)}(n),L_{i+1}^{\sigma_{i+1}(1)}(n)] \hspace*{2mm} for \hspace*{2mm} i<r
$$
We write $I_{N_{p}+q}:=I_{p}^{\sigma_{p}(q)}$, $L_{N_{p}+q} = L_{p}^{\sigma_{p}(q)}$for $0 \leq q < N$. Then for any degree $(\alpha,n)$ in the support of $M$ there is two cases:
\begin{itemize}
\item[Case I.]
if $\alpha \in I_{i}^{j}$, then $(\alpha,n)$ belongs to i-th chamber (i.e., chamber where $L_i^{j}(t)$ and $L_{i+1}^{j}(t)$ are its boundary.) of shifts $\{(\beta_1^{\sigma_j(1)}, \beta_2^{\sigma_j(1)}), \ldots, (\beta_1^{\sigma_i(j)}, \beta_2^{\sigma_i(j)})\}$
and for the other shifts $(\alpha, n)$ belongs to $(i-1)$-st chambers.\\
\item[Case II.]
if $\alpha \in I_{i}^{N}$, then  $(\alpha, n)$ belongs to $i$-th chamber
for all of the shifts.
\end{itemize}

Then by Proposition \ref{hilbert bigraded} there exist polynomials  $Q_{ij}$ such that $\dim B_{(\mu,t)} = Q_{ij}((\mu,t))$,\\
 if $\mu-td_{i} \equiv j \mod(\Delta)$.\\
By setting $\widetilde{Q}_{ik}^{j} = c_{(\beta_1^{j}, \beta_2^{j})}Q_{i,(k-\beta_1^{j}+\beta_2^{j}d_i)}(x-\beta_1^{j},y- \beta_2^{j})$ one can conclude that
if $\alpha \in I_{i}^{j},$ then
$$
\dim_k(M_{\alpha,t})= \sum_{c=1}^{j} \widetilde{Q}_{i,(\alpha-td_i)}^c (\alpha,n) + \sum_{c=j+1}^{N} \widetilde{Q}_{(i-1),(\alpha-td_{i-1})}^c (\alpha,t).
$$
\end{proof}

\begin{theorem}\label{main res}
Let $S=k[x_1, \ldots, x_n]$ be a positively graded polynomial ring over
a field $k$ and let $I$ be a homogeneous ideal in $S$.

There exist,
$t_0,m,D\in \ZZ$, linear functions $L_i(t)=a_i t+b_i$,
for $i=0,\ldots ,m$, with $a_i$ among the degrees
of the minimal generators of $I$ and $b_i\in \ZZ$, and polynomials $Q_{i,j}\in \QQ [x,y]$ for
$i=1,\ldots ,m$ and $j\in 1,\ldots ,D$, such that, for $t\geq t_0$,

(i) $L_i(t)<L_j(t)\ \Leftrightarrow\ i<j$,

(ii) If $\mu <L_0(t)$ or $\mu >L_m(t)$, then $\tor_i^S(I^t, k)_{\mu}=0$.

(iii)  If $L_{i-1} (t)\leq \mu \leq L_{i}(t)$ and
$a_i t-\mu \equiv j\mod (D)$, then
$$
\dim_k\tor_i^S(I^t, k)_{\mu}=Q_{i,j}(\mu ,t).
$$
\end{theorem}
\begin{proof}
We know from  \cite{BCH} that $M:=\tor_i^S(I^t, k)$
is a finitely generated $\ZZ^2$-graded module over $R$. Then it follows from Proposition \ref{z-graded}.
\end{proof}
Let $I=(f_1, \ldots, f_r)\subseteq S=k[x_0, \ldots, x_n]$ be a homogeneous complete intersection ideal. It is well known that minimal graded free resolution of $I$ and its power is again by koszul complex and Eagon-Northcott complex respectively, but by inspiration of above theorem  we can see that Betti table of powers of $I$ (for sufficiently large powers) encoded by finite number of numerical polynomials. In the following example we give the explicit formulas in the case the ideal is generated by three forms.
\begin{example}
Let $I=(f_1, f_2, f_3)$ be a complete intersection homogeneous ideal in the polynomial ring $S=k[x_0, \ldots, x_n]$ over a field $k$, where $\deg f_1=2$, $\deg f_2=3$ and $\deg f_3=6$. Let $R=S[T_1, T_2, T_3]$ be $\ZZ^2$-graded by setting $\deg_{\mathbb{Z}^2}(a)=(\deg a, 0)$ for any $a\in S$ and $\deg_{\mathbb{Z}^2}(T_1)=(2, 1)$,  $\deg_{\mathbb{Z}^2}(T_2)=(3, 1)$ and $\deg_{\mathbb{Z}^2}(T_3)=(6, 1)$.  Then a bigraded resolution of the Rees algebra of $I$, $\R_I$, over $R$ has the following form
$$0\rightarrow \begin{array}{c}R(-11, -2)\\
\oplus \\ R(-11, -1)\end{array}\xrightarrow{\left(\begin{array}{lll}
T_1 &T_2 &T_3\\
f_1 & f_2 & f_3
\end{array}\right)} \begin{array}{c}R(-9, -1)\\
\oplus  \\ R(-8, 1)\\
\oplus \\
R(-5, -1)\end{array}\xrightarrow{
\left(\begin{array}{l}
f_2T_3-f_3T_2\\
f_3T_1-f_1T_3\\
f_1T_2-f_2T_1
\end{array}\right)} R\rightarrow \R_I\rightarrow 0.$$
It follows that $\tor_0^S(\R_I, k)=B$, $\tor_2^S(\R_I, k)=B(-11, -1)$ and the minimal free $B$-resolution of $\tor_1^S(\R_I, k)$ is

$$0\rightarrow B(-11, -2)\xrightarrow{
\left(\begin{array}{lll}
T_1 &T_2 &T_3\\
\end{array}\right)} \begin{array}{c}B(-9, -1)\\
\oplus  \\ B(-8, 1)\\
\oplus \\
B(-5, -1)\end{array}\xrightarrow{}  \tor_1^S(\R_I, k) \rightarrow 0$$
where $B=k[T_1, T_2, T_3]$ is a non-standard graded polynomial ring over  $k$. To compute the numerical polynomials expresing the Hilbert function of $\tor$-modules of powers of $I$, we should first determine the Hilbert function of $B$. To this end, we first calculate the Hermite Normal form (HNF) of the matrix of degrees of $B$ which is
$$A=\left(\begin{array}{lll}
2&3&6\\
1&1&1\end{array}\right).$$
$$H:=HNF(A)=\left(\begin{array}{lll}
1&0&0\\
0&1&0\end{array}\right)$$ and $$U=\left(\begin{array}{lll}
-1&3&3\\
1&-2&-4\\
0&0&1\end{array}\right).$$
Here $U$ is a unimodular matrix such that $H=AU$. Now $H$ gives us a transformed  polytope $Q$, which is an interval in this case, such that the Hilbert function of $B$ at the point $(\mu, t)$ is equal to the number of lattice points on $Q$. The following inequalities give us the description of $Q$:
$$\left\{\begin{array}{r}
(\mu-3t)-3\lambda_1\leq 0,\\
(2t-\mu)+ 4\lambda_1\leq 0,\\
-\lambda_1\leq 0.\end{array}\right.$$
From Lemma \ref{chamber of n vector}, one considers two chambers $C_1$ and $C_2$ in $\ZZ^2$, defined as follows.
\begin{itemize}
\item $(\mu, t)\in C_1, \quad$ if $\left\{\begin{array}{l}
\mu-2t \geqslant 0\\
3t - \mu \geqslant 0\end{array}\right.$ \hspace{5mm}
 \item $(\mu, t)\in C_2, \quad$ if $\left\{\begin{array}{l}
6t-\mu \geqslant 0\\
\mu-3t \geqslant 0\end{array}\right.$.
\end{itemize}
The Hilbert function of $B$ written as
$$ H(B, (\mu, t))=\left\{\begin{array}{ll}
[\frac{\mu-2t}{4}] + 1&(\mu, t)\in C_1,\\\\
Ž[\frac{\mu-2t}{4}] - \frac{\mu-3t}{3} +1& (\mu, t)\in C_2    \hspace{5mm} and  \hspace{5mm} \frac{\mu-3t}{3} \in \ZZ,\\\\
Ž [\frac{\mu- 2t}{4}] - [\frac{\mu-3t}{3}]& (\mu, t)\in C_2 \hspace{5mm} and  \hspace{5mm} \frac{\mu-3t}{3} \notin \ZZ .
\end{array}\right.$$
\end{example}
The function $P(\mu,t)$ is given by polynomials depending on the value of  $\mu-2t$ mod $4$.
$$
P(\mu,t) = P_{i}(\mu,t)=\frac{\mu-2t}{4} - \frac{i}{4} + 1 \hspace{5mm} if \hspace{5mm}  \mu-2t \equiv i \mod 4$$

and the function $Q(\mu,t) $ is given by polynomials depending on the values of $\mu-3t$ mod $3$ and $\mu-2t$ mod $4$.

$$ Q(\mu,t)=\left\{\begin{array}{ll}
  \frac{6t-\mu +4j-3i}{12}  \hspace{5mm} if \hspace{5mm}  \mu-2t \equiv i \hspace{1mm},  \hspace{1mm} \mu-3t \equiv j \hspace{3mm} and  \hspace{3mm} \frac{\mu-3t}{3} \notin \ZZ,\\\\
   \frac{6t-\mu +4j-3i}{12} + 1  \hspace{5mm} if \hspace{5mm}  \mu-2t \equiv i \hspace{1mm},  \hspace{1mm} \mu-3t \equiv j \hspace{3mm} and  \hspace{3mm} \frac{\mu-3t}{3} \in \ZZ.
  \end{array}\right.$$
Now Hilbert function of the graded module $\tor_0^S(\R_I, k)$ can be written as

$$\beta_{0\mu}^{t}=\left\{\begin{array}{ll}
P(\mu,t)&2t\leqslant \mu \leqslant 3t,\\\\
Q(\mu,t)&3t< \mu \leqslant 6t\\\\
0 & Otherwise.
\end{array}\right.$$
For the Hilbert function of the graded module $\tor_1^S(\R_I, k)$ we need to write  new polynomials (see Figure  \ref{figure:Tor_1}):\\
$$
P_1 = P(\mu-5,t-1), \hspace{2mm} P_2 = P(\mu-8,t-1),  \hspace{2mm} P_3 = P(\mu-9,t-1), \hspace{2mm} P_4 = P(\mu-11,t-2), \hspace{2mm} P_5 = P(\mu-11,t-1).
$$
$$
Q_1 = Q(\mu-5,t-1), \hspace{2mm} Q_2 = Q(\mu-8,t-1),  \hspace{2mm} Q_3 = Q(\mu-9,t-1), \hspace{2mm} Q_4 = Q(\mu-11,t-2), \hspace{2mm} Q_5 = Q(\mu-11,t-1).
$$

\vspace{3mm}

$$\beta_{1\mu}^{t}=\left\{\begin{array}{ll}
P_1(\mu,t)&2t+3\leqslant \mu < 2t+6,\\\\
(P_1 +P_2)(\mu,t)&2t+6\leqslant \mu < 2t+7,\\\\
(P_1 +P_2 + P_3 - P_4)(\mu,t)&2t+7\leqslant \mu < 3t+2,\\\\
(Q_1 +P_2 + P_3 - P_4)(\mu,t)&3t+2\leqslant \mu < 3t+5,\\\\
(Q_1 +Q_2 + P_3 - P_4)(\mu,t)&3t+5\leqslant \mu < 3t+6,\\\\
(Q_1 +Q_2 + Q_3 - Q_4)(\mu,t)&3t+6\leqslant \mu < 6t-1,\\\\
(Q_2 +Q_2)(\mu,t)&6t-1\leqslant \mu < 6t+2,\\\\
Q_3(\mu,t)&6t+2< \mu \leqslant 6t+3,\\\\
0 & Otherwise.
\end{array}\right.$$
\begin{figure}
\begin{center}
\includegraphics[width=140mm]{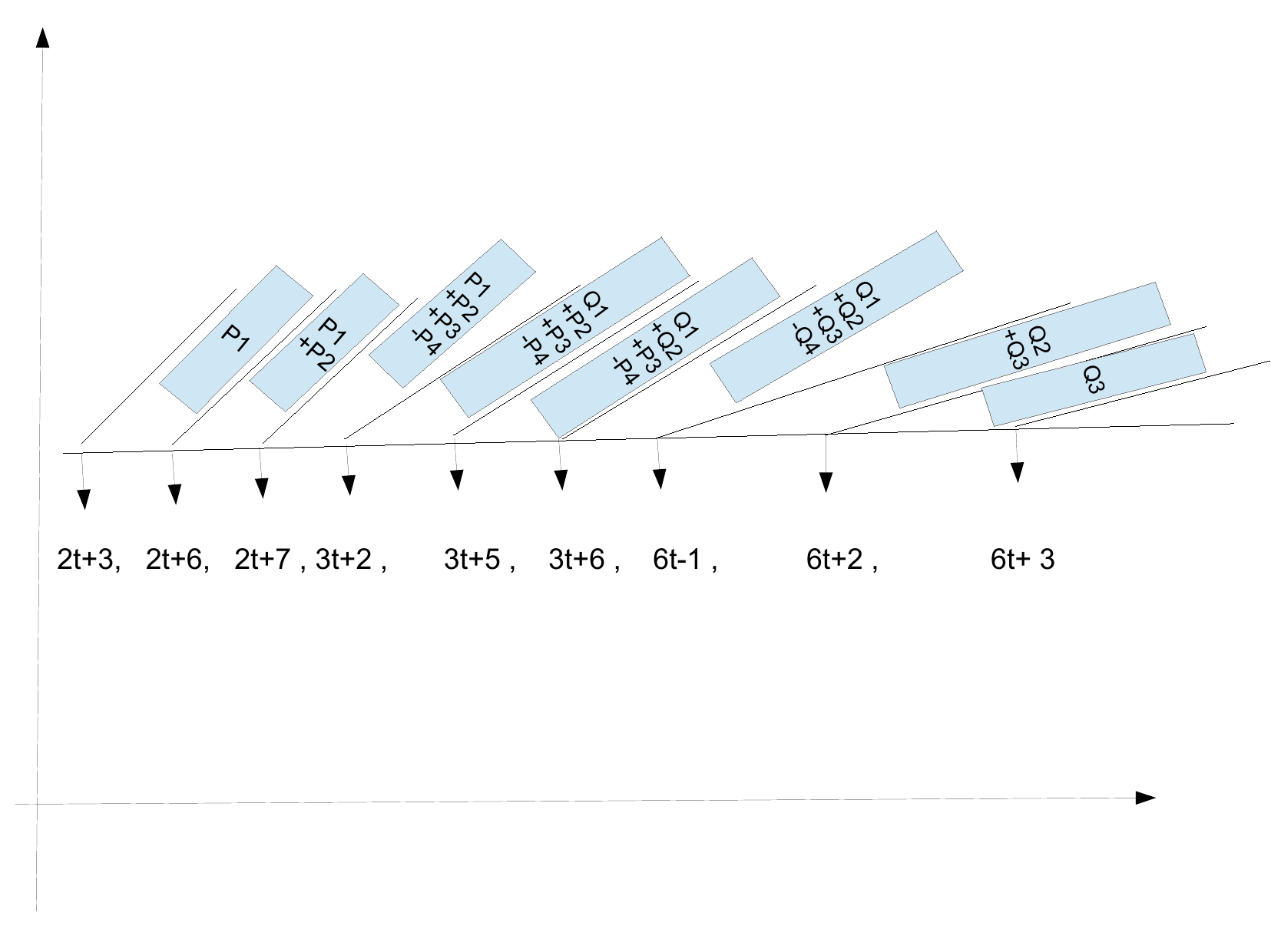}
\end{center}
  \caption{Regions  and corresponding polynomials of $\tor_1^S(\R_I, k)$.}
  \label{figure:Tor_1}
\end{figure}
Finally the Hilbert function of graded module $\tor_2^S(\R_I, k)$ can be written as

$$\beta_{2\mu}^{t}=\left\{\begin{array}{ll}
P_5(\mu,t)&2t+9\leqslant \mu < 3t+8,\\\\
Q_5(\mu,t)&3t+8\leqslant \mu \leqslant 6t+5,\\\\
0 & Otherwise.
\end{array}\right.$$

This example shows that, even in a very simple situation, quite many polynomials are involved to give a full discription of the Betti tables of powers of ideals.


\section{Acknowledgement}
This work was done as a part of the second author's Ph.D. thesis. The authors gratefully acknowledge
the support and help of Marc Chardin without whose knowledge and assistance, the present study
could not have been completed. Special thanks are due to Michele Vergne, Siamak Yassemi and Jean-Michel Kantor
 for their very useful mathematical discussions and valuable comments regarding this work.

\begin{multicols}{2}
Amir BAGHERI\\ School of Mathematics, Institute for\\   Research in Fundamental Sciences(IPM),\\  P.O. Box:19395-5746, Tehran, Iran.\\ Email: abagheri@ipm.ir\\ Marand technical college, University of\\  Tabriz, Tabriz, Iran\\
Email : a\_bageri@tabrizu.ac.ir

{Kamran LAMEI \\Institut de Math\'ematiques de Jussieu\\
UPMC,Boite th\'esard,\\4,place Jussieu,F-75252 Paris Cedex,\\France\\
Email : kamran.lamei@imj-prg.fr}
\end{multicols}

\end{document}